\documentclass{amsart}
\usepackage{amssymb,amscd,amsthm}
\usepackage{latexsym}
\date{\today}

\newcommand{\Z}{{\mathbb Z}}
\newcommand{\R}{{\mathbb R}}

\newtheorem{theorem}{Theorem}

\newtheorem{lemma}{Lemma}

\title[The Off-Diagonal Fibonacci Operator]{The Spectrum and the Spectral Type of the
Off-Diagonal Fibonacci Operator}

\author{David Damanik}

\address{Department of Mathematics, Rice University, Houston, TX~77005, USA}

\email{damanik@rice.edu}

\thanks{D.\ D.\ was supported in part by NSF grant
DMS--0653720.}

\author{Anton Gorodetski}

\address{Department of Mathematics, University of California, Irvine, CA~92697, USA}

\email{asgor@math.uci.edu}

\begin{document}

\begin{abstract}
We consider Jacobi matrices with zero diagonal and off-diagonals
given by elements of the hull of the Fibonacci sequence and show
that the spectrum has zero Lebesgue measure and all spectral
measures are purely singular continuous. In addition, if the two
hopping parameters are distinct but sufficiently close to each
other, we show that the spectrum is a dynamically defined Cantor
set, which has a variety of consequences for its local and global
fractal dimension.
\end{abstract}

\maketitle

\section{Introduction}

The study of the Fibonacci operator was initiated in the early
1980's by Kohmoto et al.\ \cite{kkt} and Ostlund et al.\
\cite{oprss}. At that point in time, the interest in this model
was based mainly on the existence of an exact renormalization
group procedure and the appearance of critical eigenstates and
zero measure Cantor spectrum. Only shortly thereafter, Shechtman
et al.\ \cite{sbgc} reported their discovery of structures, now
called quasicrystals, whose diffraction exhibits sharp Bragg peaks
with rotational symmetries that ensure the aperiodicity of the
structure in question.

The Fibonacci sequence is the central model of a quasicrystal in
one dimension. Indeed, it is aperiodic and pure point diffractive
and, moreover, it belongs to virtually all classes of mathematical
models of quasicrystals that have since been proposed. We refer to
the reader to \cite{BM} for a relatively recent account of the
mathematics related to the modelling and study of quasicrystals.

Thus, the study of the Fibonacci operator was further motivated by
the interest in electronic spectra and transport properties of
one-dimensional quasicrystals. Consequently, apart from the almost
Mathieu operator, the Fibonacci operator has been the most heavily
studied quasi-periodic operator in the last three decades; compare
the survey articles \cite{D00,D07,S95}.

Partly due to the choice of the model in the foundational papers
\cite{kkt,oprss}, the mathematical literature on the Fibonacci
operator has so far only considered the diagonal model, that is, a
discrete one-dimensional Schr\"odinger operator with potential
given by the Fibonacci sequence. Given the connection to
quasicrystals and hence aperiodic point sets, and particularly
cut-and-project sets, it is however equally (if not more) natural
to study the off-diagonal model. That is, one considers a Jacobi
matrix with zero diagonal and off-diagonals given by the Fibonacci
sequence. This would correspond to hopping over unequal distances,
the sites being points on the real line generated by the
cut-and-project scheme that produces the Fibonacci point set.
While we will not explain these concepts here in detail, we do
refer the reader to \cite{BM} for background and we will describe
the resulting model carefully in the next section. For further
motivation to study the off-diagonal model, we mention that it has
been the object of interest in a number of physics papers; see,
for example, \cite{EL06,EL08,kst,RP,VP,YYXZZ}.

In fact, our interest in the off-diagonal model was triggered by the
numerical results for separable higher-dimensional models obtained
recently by Even-Dar Mandel and Lifshitz in \cite{EL06,EL08}. We
intend to explore the observations reported in these papers and
prove some of the phenomena rigorously, and for this task, we will
need the basic spectral results for the off-diagonal Fibonacci
operator obtained in this paper.

The paper is organized as follows. In Section~\ref{s.s1} we
describe in detail the model we study and the results we obtain
for it. The results concerning the spectrum as a set are proved in
Section~\ref{s.s2} and the results concerning the spectral type
are proved in Section~\ref{s.s3}.

\medskip

\noindent\textit{Acknowledgment.} We wish to thank Serge Cantat
for useful conversations.

\section{Model and Results}\label{s.s1}

In this section we describe the off-diagonal Fibonacci operator and the results for it that are
proved in subsequent sections. Let $a,b$ be two positive real numbers and consider the Fibonacci
substitution,
$$
S(a) = ab , \quad S(b) = a.
$$
This substitution rule extends to finite and one-sided infinite words by concatenation. For example,
$S(aba) = abaab$. Since $S(a)$ begins with $a$, one obtains a one-sided infinite sequence that is
invariant under $S$ by iterating the substitution rule on $a$ and taking a limit. Indeed, we have
\begin{equation}\label{f.fibrec}
S^k(a) = S^{k-1}(S(a)) = S^{k-1}(ab) = S^{k-1}(a) S^{k-1}(b) = S^{k-1}(a) S^{k-2}(a).
\end{equation}
In particular, $S^k(a)$ starts with $S^{k-1}(a)$ and hence there is
a unique one-sided infinite sequence $u$, the so-called Fibonacci
substitution sequence, that starts with $S^k(a)$ for every $k$. The
hull $\Omega_{a,b}$ is then obtained by considering all two-sided
infinite sequences that locally look like $u$,
$$
\Omega_{a,b} = \{ \omega \in \{a,b\}^\Z : \text{ every subword of
$\omega$ is a subword of } u \}.
$$
It is known that, conversely, every subword of $u$ is a subword of
every $\omega \in \Omega_{a,b}$. In this sense, $u$ and all elements
of the hull $\omega$ look exactly the same locally.

We wish to single out a special element of $\Omega_{a,b}$. Notice
that $ba$ occurs in $u$ and that $S^2(a) = aba$ begins with $a$
and $S^2(b) = ab$ ends with $b$. Thus, iterating $S^2$ on $b | a$,
where $|$ denotes the eventual origin, we obtain as a limit a
two-sided infinite sequence which belongs to $\Omega_{a,b}$ and
coincides with $u$ to the right of the origin. This element of
$\Omega_{a,b}$ will be denoted by $\omega_s$.

Each $\omega \in \Omega_{a,b}$ generates a Jacobi matrix $H_\omega$
acting in $\ell^2(\Z)$,
$$
(H_\omega \psi)_n = \omega_{n+1} \psi_{n+1} + \omega_{n}
\psi_{n-1}.
$$
With respect to the standard orthonormal basis $\{ \delta_n \}_{n
\in \Z}$ of $\ell^2(\Z)$, where $\delta_n$ is one at $n$ and
vanishes otherwise, this operator has the matrix
$$
\begin{pmatrix}
\ddots & \ddots & \ddots & {} & {} & {} \\
\ddots & 0 & \omega_{-1} & 0 & {} & {} \\
\ddots & \omega_{-1} & 0 & \omega_0 & 0 & \vphantom{\ddots} \\
{} & 0 & \omega_0 & 0 & \omega_1 & \ddots  \\
{} & {} & 0 & \omega_1 & 0 & \ddots   \\
{} & {} & {} & \ddots & \ddots & \ddots
\end{pmatrix}
$$
and it is clearly self-adjoint.

This family of operators, $\{H_\omega\}_{\omega \in \Omega_{a,b}}$,
is called the off-diagonal Fibonacci model. Of course, the structure
of the Fibonacci sequence disappears when $a = b$. In this case, the
hull consists of a single element, the constant two-sided infinite
sequence taking the value $a = b$, and the spectrum and the spectral
measures of the associated operator $H_\omega$ are well understood.
For this reason, we will below always assume that $a \not= b$.
Nevertheless, the limiting case, where we fix $a$, say, and let $b$
tend to $a$ is of definite interest.

Our first result concerns the spectrum of $H_\omega$. For $S \subset
\R$, we denote by $\text{dim}_H S$ the Hausdorff dimension of $S$
and by $\text{dim}_B S$ the box counting dimension of $S$ (which is
then implicitly claimed to exist); see \cite{DEGT} or \cite{PT} for
the definition of these fractal dimensions.

\begin{theorem}\label{t.main1}
Suppose $a,b > 0$ and $a \not= b$. Then, there exists a compact
set $\Sigma_{a,b} \subset \R$ such that $\sigma(H_\omega) =
\Sigma_{a,b}$ for every $\omega \in \Omega_{a,b}$, and
\begin{itemize}

\item[{\rm (i)}] $\Sigma_{a,b}$ has zero Lebesgue measure.

\item[{\rm (ii)}] The Hausdorff dimension $\dim_H \Sigma_{a,b}$ is
an analytic function of $a$ and $b$.

\item[{\rm (iii)}]  $0<\dim_H \Sigma_{a,b}<1$.
\end{itemize}

Also,  there exists $\varepsilon_0>0$ such that if $\frac{a^2 +
b^2}{2ab} < 1 + \varepsilon_0$ {\rm (}in other words, if $a$ and
$b$ are close enough{\rm )}, then

\begin{itemize}

\item[{\rm (iv)}] The spectrum $\Sigma_{a,b}$ is a Cantor set that depends
continuously on $a$ and $b$ in the Hausdorff metric.

\item[{\rm (v)}] For every small $\delta>0$ and every $E \in
\Sigma_{a,b}$, we have
\begin{align*}
\dim_H \left( (E - \delta, E + \delta) \cap \Sigma_{a,b} \right) & =
\dim_B \left( (E - \delta, E + \delta) \cap \Sigma_{a,b} \right) \\
& = \dim_H \Sigma_{a,b}  =\dim_B \Sigma_{a,b}.
\end{align*}

\item[{\rm (vi)}] Denote $\alpha = \dim_H \Sigma_{a,b}$, then the
Hausdorff $\alpha$-measure of $\Sigma_{a,b}$ is positive and
finite.
\end{itemize}
\end{theorem}

Properties (ii)--(vi) are related to hyperbolicity of the trace
map described in the next section, see \cite{Can, Cas, DG}. It is
reasonable to expect that in fact the properties (iv)--(vi) for
$\Sigma_{a,b}$ hold for all positive $a\ne b$, but one of the
steps in the proof is currently missing as will be explained in
Section~\ref{s.s2}.

Next, we turn to the spectral type of $H_\omega$.

\begin{theorem}\label{t.main2}
Suppose $a,b > 0$ and $a \not= b$. Then, for every $\omega \in
\Omega_{a,b}$, $H_\omega$ has purely singular continuous spectrum.
\end{theorem}

Throughout the rest of the paper we will only consider $a,b > 0$
with $a \not= b$. Theorem~\ref{t.main1} is proved in
Section~\ref{s.s2} and Theorem~\ref{t.main2} is proved in
Section~\ref{s.s3}.

\section{The Trace Map and its Relation to the Spectrum}\label{s.s2}

The spectral properties of $H_\omega$ are closely related to the
behavior of the solutions to the difference equation
\begin{equation}\label{f.eve}
\omega_{n+1} u_{n+1} + \omega_{n} u_{n-1} = E u_n.
\end{equation}
Denote
$$
U_n = \begin{pmatrix} u_n \\ \omega_{n} u_{n-1} \end{pmatrix}.
$$
Then $u$ solves \eqref{f.eve} (for every $n \in \Z$) if and only if $U$ solves
\begin{equation}\label{f.tmatrix}
U_{n} = T_\omega(n,E) U_{n-1},
\end{equation}
(for every $n \in \Z$), where
$$
T_\omega(n,E) = \frac{1}{\omega_{n}} \begin{pmatrix} E & -1 \\
\omega_{n}^2 & 0 \end{pmatrix}.
$$
Note that $\det T_\omega(n,E) = 1$. Iterating \eqref{f.tmatrix}, we find
$$
U_{n} = M_\omega(n,E) U_0,
$$
where
$$
M_\omega(n,E) = T_\omega(n,E) \times \cdots T_\omega(1,E).
$$
With the Fibonacci numbers $\{F_k\}$, generated by $F_0 = F_1 = 1$, $F_{k+1} = F_k + F_{k-1}$
for $k \ge 1$, we define
$$
x_k = x_k(E) = \frac12 \mathrm{Tr} M_{\omega_s}(F_k,E).
$$
For example, we have
\begin{align*}
M_{\omega_s}(F_1,E) & = \frac{1}{a} \begin{pmatrix} E & -1 \\ a^2 & 0 \end{pmatrix} \\
M_{\omega_s}(F_2,E) & = \frac{1}{b} \begin{pmatrix} E & -1 \\ b^2 & 0 \end{pmatrix} \frac{1}{a}
\begin{pmatrix} E & -1 \\ a^2 & 0 \end{pmatrix} = \frac{1}{ab} \begin{pmatrix} E^2 - a^2 & -E
\\ E b^2 & - b^2 \end{pmatrix} \\
M_{\omega_s}(F_3,E) & = \frac{1}{a} \begin{pmatrix} E & -1 \\ a^2 & 0 \end{pmatrix} \frac{1}{b}
\begin{pmatrix} E & -1 \\ b^2 & 0 \end{pmatrix} \frac{1}{a} \begin{pmatrix} E & -1 \\ a^2 & 0
\end{pmatrix} = \frac{1}{a^2 b} \begin{pmatrix} E^3 - E a^2 - E b^2 & -E^2 + b^2 \\ E^2 a^2 -
a^4 & - E a^2 \end{pmatrix}
\end{align*}
and hence
\begin{equation}\label{f.tmic}
x_1 = \frac{E}{2a} , \quad x_2 = \frac{E^2 - a^2 - b^2}{2ab},
\quad x_3 = \frac{E^3 - 2 E a^2 - E b^2}{2a^2 b}.
\end{equation}

\begin{lemma}\label{l.l1}
We have
\begin{equation}\label{f.tmrec}
x_{k+1} = 2 x_k x_{k-1} - x_{k-2}
\end{equation}
for $k \ge 2$. Moreover, the quantity
\begin{equation}\label{f.tminv}
I_k = x_{k+1}^2 + x_k^2 + x_{k-1}^2 - 2 x_{k+1} x_k x_{k-1} - 1
\end{equation}
is independent of both $k$ and $E$ and it is given by
$$
I = \frac{(a^2 + b^2)^2}{4a^2b^2} - 1.
$$
\end{lemma}

\begin{proof}
Since $\omega_s$ restricted to $\{n \ge 1\}$ coincides with $u$
and the prefixes $s_k$ of $u$ of length $F_k$ obey $s_{k+1} = s_k
s_{k-1}$ for $k \ge 2$ by construction (cf.~\eqref{f.fibrec}), the
recursion \eqref{f.tmrec} follows as in the diagonal case; compare
\cite{D00,D07,S87}. This recursion in turn implies readily that
\eqref{f.tminv} is $k$-independent.

In particular, the $x_k$'s are again generated by the trace map
$$
T(x,y,z) = (2xy - z , x , y)
$$
and the preserved quantity is again
$$
I(x,y,z) = x^2 + y^2 + z^2 - 2xyz - 1.
$$
The only difference between the diagonal and the off-diagonal model can be found in the initial
conditions. How are $x_1, x_0, x_{-1}$ obtained? Observe that the trace map is invertible and
hence we can apply its inverse twice to the already defined quantity $(x_3,x_2,x_1)$. We have
$$
T^{-1}(x,y,z) = (y,z, 2yz - x)
$$
and hence, using \eqref{f.tmic},
\begin{align*}
(x_1,x_0,x_{-1}) & = T^{-2} ( x_3 , x_2 , x_1 ) \\
& = T^{-2} \left( \frac{E^3 - 2 E a^2 - E b^2}{2a^2 b} , \frac{E^2 - a^2 - b^2}{2ab}, \frac{E}{2a}
\right) \\
& = T^{-1} \left( \frac{E^2 - a^2 - b^2}{2ab}, \frac{E}{2a} , 2 \frac{(E^2 - a^2 - b^2)E}{4a^2 b} -
\frac{E^3 - 2 E a^2 - E b^2}{2a^2 b} \right) \\
& = T^{-1} \left( \frac{E^2 - a^2 - b^2}{2ab}, \frac{E}{2a} , \frac{E}{2b} \right) \\
& = \left( \frac{E}{2a} , \frac{E}{2b} , 2 \frac{E^2}{4ab} - \frac{E^2 - a^2 - b^2}{2ab} \right) \\
& = \left( \frac{E}{2a} , \frac{E}{2b} , \frac{a^2 + b^2}{2ab} \right)
\end{align*}
It follows that
\begin{align*}
I(x_{k+1} , x_k , x_{k-1}) & = I(x_1 , x_0 , x_{-1}) \\
& = \frac{E^2}{4a^2} + \frac{E^2}{4b^2} + \frac{(a^2 + b^2)^2}{4a^2b^2} - 2 \frac{E^2(a^2 + b^2)}{8a^2b^2}
- 1 \\
& = \frac{(a^2 + b^2)^2}{4a^2b^2} - 1
\end{align*}
for every $k \ge 0$.
\end{proof}

It is of crucial importance for the spectral analysis that, as in
the diagonal case, the invariant is energy-independent and
strictly positive when $a \not= b$!

\begin{lemma}\label{l.l2}
The spectrum of $H_\omega$ is independent of $\omega$ and may be
denoted by $\Sigma_{a,b}$. With
$$
\sigma_k = \{ E \in \R : |x_k| \le 1 \},
$$
we have
\begin{equation}\label{f.descspec}
\Sigma_{a,b} = \bigcap_{k \ge 1} \sigma_k \cup \sigma_{k+1}.
\end{equation}
Moreover, for every $E \in \Sigma_{a,b}$ and $k \ge 2$,
\begin{equation}\label{t.tmbound}
|x_k| \le 1 + \left( \frac{(a^2 + b^2)^2}{4a^2b^2} - 1 \right)^{1/2}
\end{equation}
and for $E \not\in \Sigma_{a,b}$, $|x_k|$ diverges
super-exponentially.
\end{lemma}

\begin{proof}
It is well known that the hull $\Omega_{a,b}$ together with the
standard shift transformation is minimal. In particular, every
$\omega \in \Omega_{a,b}$ may be approximated pointwise by a
sequence of shifts of any other $\tilde \omega \in \Omega_{a,b}$.
The associated operators then converge strongly and we get
$\sigma(H_\omega) \subseteq \sigma(H_{\tilde \omega})$. Reversing
the roles of $\omega$ and $\tilde \omega$, the first claim follows.

So let $\Sigma_{a,b}$ denote the common spectrum of the operators
$H_\omega$, $\omega \in \Omega_{a,b}$. We have $\|H_\omega\| \le
\max \{ 2a,2b \}$. Thus, $\Sigma_{a,b} \subseteq [- \max \{ 2a,2b \}
, \max \{ 2a,2b \} ] =: I_{a,b}$. For $E \in I_{a,b}$, we have that
at least one of $x_1$, $x_0$ belongs to $[-1,1]$. This observation
allows us to use the exact same arguments S\"ut\H{o} used to prove
\eqref{f.descspec} for the diagonal model in \cite{S87}.

The only point where care needs to be taken is the claim
that $\sigma_k$ is the spectrum of the periodic Jacobi
matrix obtained by repeating the values $\omega_s$ takes on $\{ 1
\le n \le F_k \}$ periodically on the off-diagonals. This,
however, follows from the general theory of periodic Jacobi
matrices, which relies on the diagonalization of the monodromy
matrix (which is $M_{\omega_s}(F_k,E)$ in this case) in order to
obtain Floquet solutions and in particular discriminate between
those energies that permit exponentially growing solutions and
those that do not. This distinction works just as well here, but
one needs to use that the $\omega_n$'s that enter in the $U_n$'s
are uniformly bounded away from zero and infinity.

Thus, after paying attention to this fact, we may now proceed along
the lines of S\"ut\H{o}. Let us describe the main steps of the
argument. Since at least one of $x_1$, $x_0$ belongs to $[-1,1]$, we
have a result analogous to \cite[Lemma~2]{S87} with the same proof
as given there. Namely, the sequence $\{x_k\}_{k \ge 0}$ is
unbounded if and only if there exists $k$ such that $|x_{k}|> 1$ and
$|x_{k+1}| > 1$. Moreover, we then have $|x_{k + l}| > c^{F_l}$ for
some $c > 1$ and all $l \ge 0$. This shows
$$
\sigma_k \cup \sigma_{k+1} = \bigcup_{l \ge 0} \sigma_{k + l}.
$$
Using now the fact that the $F_k$ periodic Jacobi matrices with
spectrum $\sigma_k$ converge strongly to $H_{\omega_s}$, we obtain
$$
\Sigma_{a,b} \subseteq \bigcap_{k \ge 1} \overline{\bigcup_{l \ge 0}
\sigma_{k + l}} = \bigcap_{k \ge 1} \overline{\sigma_k \cup
\sigma_{k+1}} = \bigcap_{k \ge 1} \sigma_k \cup \sigma_{k+1},
$$
since the spectra $\sigma_k$ and $\sigma_{k+1}$ are closed sets.
Thus, we have one inclusion in \eqref{f.descspec}.

Next, suppose $E \in \bigcap_{k \ge 1} \sigma_k \cup \sigma_{k+1}$
If $k \ge 1$ is such that $|x_k| > 1$, then $|x_{k-1}| \le 1$ and
$|x_{k+1}| \le 1$. Since we have
$$
x_{k+1}^2 + x_k^2 + x_{k-1}^2 - 2 x_{k+1} x_k x_{k-1} -1 =
\frac{(a^2 + b^2)^2}{4a^2b^2} - 1 ,
$$
this implies
$$
x_k = x_{k+1} x_{k-1} \pm \left( 1 - x_{k+1}^2 - x_{k-1}^2 +
x_{k+1}^2 x_{k-1}^2 + \frac{(a^2 + b^2)^2}{4a^2b^2} - 1
\right)^{1/2}
$$
and hence
$$
|x_k| \le |x_{k+1} x_{k-1}| + \left( (1 - x_{k+1}^2) (1 -
x_{k-1}^2) + \left( \frac{(a^2 + b^2)^2}{4a^2b^2} - 1 \right)
\right)^{1/2}
$$
which, using $|x_{k-1}| \le 1$ and $|x_{k+1}| \le 1$ again, implies
the estimate \eqref{t.tmbound} for $E \in \bigcap_{k \ge 1} \sigma_k
\cup \sigma_{k+1}$. We will show in the next section that the
boundedness of the sequence $\{x_k\}_{k \ge 0}$ implies that, for
arbitrary $\omega \in \Omega_{a,b}$, no solution of the difference
equation \eqref{f.eve} is square-summable at $+\infty$.
Consequently, such $E$'s belong to $\Sigma_{a,b}$.\footnote{This
follows by a standard argument: If $E \not\in \Sigma_{a,b}$, then
$(H_\omega - E)^{-1}$ exists and hence $(H_\omega-E)^{-1} \delta_0$
is an $\ell^2(\Z)$ vector that solves \eqref{f.eve} away from the
origin. Choosing its values for $n \ge 1$, say, and then using
\eqref{f.eve} to extend it to all of $\Z$, we obtain a solution that
is square-summable at $+\infty$.} This shows the other inclusion in
\eqref{f.descspec} and hence establishes it. Moreover, it follows
that \eqref{t.tmbound} holds for every $E \in \Sigma_{a,b}$.

Finally, from the representation \eqref{f.descspec} of
$\Sigma_{a,b}$ and our observation above about unbounded sequences
$\{x_k\}_{k \ge 0}$, we find that $|x_k|$ diverges
super-exponentially for $E \not\in \Sigma_{a,b}$. This concludes
the proof of the lemma.
\end{proof}

\begin{lemma}\label{l.l3}
For every $E \in \R$, there is $\gamma(E) \ge 0$ such that
$$
\lim_{n \to \infty} \frac{1}{n} \log \| M_\omega(n,E) \| =
\gamma(E),
$$
uniformly in $\omega \in \Omega_{a,b}$.
\end{lemma}

\begin{proof}
This follows directly from the uniform subadditive ergodic theorem;
compare \cite{DL99,DL06,H93,L02}.
\end{proof}

\begin{lemma}\label{l.l4}
The set $\mathcal{Z}_{a,b} := \{ E \in \R : \gamma(E) = 0 \}$ has
zero Lebesgue measure.
\end{lemma}

\begin{proof}
This is one of the central results of Kotani theory; see
\cite{K89} and also \cite{D07b}. Note that these papers only
discuss the diagonal model. Kotani theory for Jacobi matrices is
discussed in Carmona-Lacroix \cite{CL90} and the result needed can
be deduced from what is presented there. For a recent reference
that states a result sufficient for our purpose explicitly, see
Remling \cite{R07}.
\end{proof}

\begin{lemma}\label{l.l5}
We have $\Sigma_{a,b} = \mathcal{Z}_{a,b}$.
\end{lemma}

\begin{proof}
The inclusion $\Sigma_{a,b} \supseteq \mathcal{Z}_{a,b}$ holds by
general principles. For example, one can construct Weyl sequences by
truncation when $\gamma(E) = 0$. The inclusion $\Sigma_{a,b}
\subseteq \mathcal{Z}_{a,b}$ can be proved in two ways. Either one
uses the boundedness of $x_k$ for energies $E \in \Sigma_{a,b}$ to
prove explicit polynomial upper bounds for $\|M_\omega(n,E)\|$ (as
in \cite{IT91} for $\omega = \omega_s$ or in \cite{DL99} for general
$\omega \in \Omega_{a,b}$), or one combines the proof of the absence
of decaying solutions at $+\infty$ for $E \in \Sigma_{a,b}$ given in
the next section with Osceledec's Theorem, which states that
$\gamma(E) > 0$ would imply the existence of an exponentially
decaying solution at $+\infty$. Here we use one more time that $U_n$
is comparable in norm to $(u_n,u_{n-1})^T$.
\end{proof}

\begin{proof}[Proof of Theorem~\ref{t.main1}.]
The existence of the uniform spectrum $\Sigma_{a,b}$ was shown in
Lemma~\ref{l.l3} and the fact that $\Sigma_{a,b}$ has zero
Lebesgue measure follows from Lemmas~\ref{l.l4} and \ref{l.l5}.
The set of bounded orbits of the restriction of the trace map
$T:\mathbb{R}^3\to \mathbb{R}^3$ to the invariant surface
$I(x,y,z)=C\equiv\frac{(a^2 + b^2)^2}{4a^2b^2} - 1, C>0,$ is
hyperbolic; see \cite{Can} (and also \cite{DG} for $C$
sufficiently small and \cite{Cas} for $C$ sufficiently large). Due
to Lemma~\ref{l.l2}, the points of the spectrum correspond to the
points of the intersection of the line of the initial conditions
$$
l_{a,b}\equiv\left\{ \left( \frac{E}{2a} , \frac{E}{2b} , \frac{a^2
+ b^2}{2ab} \right) : E \in \mathbb{R}\right\}
$$
with the stable manifolds of the hyperbolic set of bounded orbits.
Properties~(ii) and (iii) can be proved in exactly the same way as
Theorem~6.5 in \cite{Can}. The line $l_{a,b}$ intersects the
stable lamination of the hyperbolic set transversally for
sufficiently small $C > 0$, as can be shown in the same way as for
the diagonal Fibonacci Hamiltonian with a small coupling constant;
see \cite{DG}. Therefore the spectrum $\Sigma_{a,b}$ for close
enough $a$ and $b$ is a dynamically defined Cantor set, and the
properties (iv)--(vi) follow; see \cite{DEGT, DG, Ma, MM, P, PT}
and references therein.
\end{proof}

Notice that a proof of the transversality of the line $l_{a,b}$ to
the stable lamination of the hyperbolic set of bounded orbits for
arbitrary $a\ne b$ would imply the properties (iv)--(vi) for these
values of $a$ and $b$.

\section{Singular Continuous Spectrum}\label{s.s3}

In this section we prove Theorem~\ref{t.main2}. Given the results from the previous section, we can
follow the proofs from the diagonal case quite closely.

\begin{proof}[Proof of Theorem~\ref{t.main2}.]
Since the absence of absolutely continuous spectrum follows from
zero measure spectrum, we only need to show the absence of point
spectrum. It was shown by Damanik and Lenz \cite{DL} that, given any
$\omega \in \Omega_{a,b}$ and $k \ge 1$, the restriction of $\omega$
to $\{ n \ge 1 \}$ begins with a square
$$
\omega_1 \ldots \omega_{2F_k} \ldots = \omega_1 \ldots \omega_{F_k} \omega_1 \ldots \omega_{F_k} \ldots
$$
such that $\omega_1 \ldots \omega_{F_k}$ is a cyclic permutation
of $S^k(a)$. By cyclic invariance of the trace, it follows that
$\mathrm{Tr} M_\omega(F_k,E) = 2 x_k(E)$ for every $E$.

The Cayley-Hamilton Theorem, applied to $M_\omega(F_k,E)$, says
that
$$
M_\omega(F_k,E)^2 - \left( \mathrm{Tr} M_\omega(F_k,E) \right)
M_\omega(F_k,E) + I = 0,
$$
which, by the observations above, translates to
$$
M_\omega(2F_k,E) - 2x_k M_\omega(F_k,E) + I = 0.
$$

If $E \in \Sigma_{a,b}$ and $u$ is a solution of the difference
equation \eqref{f.eve}, it therefore follows that
$$
U(2F_k + 1) - 2x_k U(F_k + 1) + U(1) = 0.
$$
If $u$ does not vanish identically, this shows that $u_n \not\to
0$ as $n \to \infty$ since the $x_k$'s are bounded above and the
$\omega_n$'s are bounded below away from zero. In particular, if
$E \in \Sigma_{a,b}$, then no non-trivial solution of
\eqref{f.eve} is square-summable at $+\infty$ and hence $E$ is not
an eigenvalue. It follows that the point spectrum of $H_\omega$ is
empty.
\end{proof}

\end{document}